\documentclass[preprint,12pt]{elsarticle}
\usepackage[cp1251]{inputenc}

\usepackage{a4wide,amsfonts, amsmath, amscd}
\usepackage{graphicx}
\usepackage{epsfig}

\usepackage{amssymb}
\usepackage{amsthm}

\newcommand{\sysn}{\left\{\begin{array}{rcl}}
\newcommand{\sysk}{\end{array}\right.}

\newtheorem{theorem}{Theorem}[section]

\newtheorem{proposition}[theorem]{Proposition}
\theoremstyle{definition}

\newtheorem{corollary}[theorem]{Corollary}

\journal{...}

\begin{document}

\begin{frontmatter}




\title{On the first Banach problem, concerning condensations of absolute $\kappa$-Borel sets onto
compacta}

\author{Alexander V. Osipov}

\ead{OAB@list.ru}


\address{Krasovskii Institute of Mathematics and Mechanics, Ural Federal
 University, Yekaterinburg, Russia}

\begin{abstract}
 It is consistent that the continuum be arbitrary large and no absolute $\kappa$-Borel set $X$ of
density $\kappa$, $\aleph_1<\kappa<\mathfrak{c}$, condenses onto a compactum.

It is consistent that the continuum be arbitrary large and any absolute $\kappa$-Borel set $X$ of
density $\kappa$, $\kappa\leq\mathfrak{c}$, containing a
closed subspace of the Baire space of weight $\kappa$, condenses
onto a compactum.

In particular, applying Brian's results in model theory, we get the following unexpected result.
Given any $A\subseteq \mathbb{N}$ with $1\in A$, there is a forcing extension in which every  absolute $\aleph_n$-Borel set, containing a closed subspace of the Baire space
of weight $\aleph_n$, condenses onto a compactum if, and only if, $n\in A$.

\end{abstract}

\begin{keyword} condensation \sep compact metric space \sep Banach
Problem

\MSC[2020] 57N17 \sep 57N20 \sep 54C10 \sep 54E99

\end{keyword}

\end{frontmatter}

\section{Introduction}

The following well-known problem of Stefan Banach from the Scottish Book:

{\bf Banach Problem.} When does a metric (possibly Banach) space
$X$ admit a condensation (i.e. a bijective continuous mapping)
onto a compactum (= compact metric space)?

\medskip
The complete answer to the Banach Problem  for Banach spaces can
be found in \cite{os}.

\medskip

In \cite{os1}, it is proved that any metric
space of weight $\mathfrak{c}$ admits a condensation onto the
Hilbert cube.

\medskip

In 1976, E.G. Pytkeev \cite{pyt} proved the following remarkable theorem for separable absolute $\aleph_0$-Borel sets.

\begin{theorem}(Pytkeev)
Every separable absolute Borel space $X$ condenses onto the Hilbert cube, whenever $X$ is not $\sigma$-compact.
\end{theorem}

 Hurewicz proved that a Polish space fails to be $\sigma$-compact if and only if it contains a closed set homeomorphic to $\omega^{\omega}$. Note that if a space $X$  contains a closed set homeomorphic to $\omega^{\omega}$ then $X$ is not $\sigma$-compact. Thus, by Pytkeev's Theorem, every separable absolute Borel space $X$ condenses (i.e. bijectively continuously mapped) onto the Hilbert cube, whenever $X$ contains a closed set homeomorphic to $\omega^{\omega}$.  This fact motivates us to consider the Banach Problem for absolute $\kappa$-Borel sets containing a closed set homeomorphic to $\kappa^{\omega}$.

\medskip

Note that any absolute $\kappa$-Borel set $X$ of density $\kappa$,
containing a closed subspace of the Baire space $\kappa^{\omega}$
of weight $\kappa$, condenses onto a Banach space of weight
$\kappa$ (Theorem 3 in \cite{Med}). This entails the following result.

\medskip

{\it Any absolute $\aleph_1$-Borel sets, containing a closed set homeomorphic to $\omega_1^{\omega}$,  condenses
onto a compactum}.

\medskip

It is enough to note that, by the Banakh-Plichko Theorem \cite{BaPl}, any Banach space of density $\aleph_1$ condenses onto the Hilbert cube.

\medskip

By Theorem 2.3 in \cite{os}, {\it there is a forcing
extension in which any Banach space of weight $\kappa$ ($\aleph_1<\kappa< \mathfrak{c}$) condenses onto a compactum}.

\medskip

Thus, combining the previous facts, we get the following result.

\begin{theorem}
It is consistent that the continuum be arbitrary large and any absolute $\kappa$-Borel set $X$ of
density $\kappa$, containing a closed subspace of the Baire space
of weight $\kappa$ ($\kappa\leq\mathfrak{c}$), condenses
onto a compactum.
\end{theorem}

\medskip

In this paper, we continue to study of the Banach Problem in the
class of absolute $\kappa$-Borel sets. The following main result is proved.

\medskip
\begin{theorem}\label{th3}\label{th3} It is consistent that the continuum be arbitrary
large and no absolute
$\kappa$-Borel set $X$ of density $\kappa$,
$\aleph_1<\kappa<\mathfrak{c}$, condenses onto a compactum.
\end{theorem}

The proof of the main theorem can be summarized as follows. If there is an absolute
$\kappa$-Borel set of density $\kappa$ that condenses onto a compactum $K$, then
there is a covering $\mathcal{C}$ of $K$ with $\leq cf[\kappa]^{\omega}$
Borel sets for which no $<\kappa$-sized
subset of $\mathcal{C}$ covers $K$. But there is a model (essentially the Cohen model)
in which $cf[\kappa]^{\omega}<\mathfrak{c}$ for every $\kappa<\mathfrak{c}$, and if there is a covering of $K$ with $<\mathfrak{c}$
Borel sets then some $\aleph_1$-many of those Borel sets already cover $K$. This is
a contradiction if $\kappa\in (\aleph_1, \mathfrak{c})$, i.e. $\aleph_1<\kappa<\mathfrak{c}$.

\section{Preliminaries}

The {\it density} $d(X)$ of a topological space $X$ is the
smallest cardinality of a dense subset of $X$. The cardinal
function $w(X)$ is the weight of $X$, which is defined by
$w(X)=min\{|\mathcal{B}| : \mathcal{B}$ is a base for $X\} +
\omega$. For a metrizable space $X$, we have $d(X)=w(X)$.

 Since metrizable
compact spaces have cardinality at most continuum, every metric
space admitting a condensation onto a compactum has density at
most continuum.

\medskip

Recall that the family of {\it hyper-Borel sets} of $X$, denoted
$HB(X)$, to be the smallest family of subsets of $X$ which
contains (i) the open sets of $X$, (ii) $X\setminus B$ whenever
$B\in HB(X)$, and (iii) $\bigcup_t B_t$ whenever each $B_t\in
HB(X)$ and $\{B_t\}$ is a $\sigma$-discrete family of subsets of
$X$. A set $B\subset X$ is called {\it $\kappa$-Borel} if $B$ is
hyper-Borel in $X$ of class $\alpha$ for some ordinal $\alpha$ of
cardinal $\leq \kappa$. The $\aleph_0$-Borel sets are the ordinary
Borel sets \cite{Han}.

A space $X$ is called {\it absolute $\kappa$-Borel}, if $X$ is
homeomorphic to a $\kappa$-Borel subset of some complete
metrizable space. Thus, absolute $\aleph_0$-Borel sets are the
ordinary absolute Borel sets.

\medskip
If $\kappa$ is an infinite cardinal, then $[\kappa]^{\mu}$ denotes the set of all subsets of $\kappa$ with cardinality~$\mu$. A family $\mathcal{F}\subseteq [\kappa]^{\omega}$ is {\it cofinal} if for every $A\in [\kappa]^{\omega}$ there is some $B\in \mathcal{F}$ such that $B\supseteq A$. Let $cf[\kappa]^{\omega}=\min\{|\mathcal{F}|: \mathcal{F}$ is a cofinal subset of $[\kappa]^{\omega}\}$.

\medskip

Let $FIN(\kappa,2)$ be the partial order of finite partial
functions from $\kappa$ to $2$, i.e., Cohen forcing.

\begin{proposition}(Corollary 3.13 in \cite{BrMi})\label{prop} Suppose $M$ is
a countable transitive model of $ZFC+GCH$. Let $\kappa$ be any
cardinal of $M$ of uncountable cofinality which is not the
successor of a cardinal of countable cofinality.  Suppose that $G$
is $FIN(\kappa,2)$-generic over $M$, then in $M[G]$ the continuum
is $\kappa$ and for every uncountable $\gamma<\kappa$ if $F:
\gamma^{\omega}\rightarrow \omega^{\omega}$ is continuous and
onto, then there exists a $Q\in [\gamma]^{\omega_1}$ such that
$F(Q^{\omega})=\omega^{\omega}$.
\end{proposition}

\section{Main result}

Let us prove the main Theorem \ref{th3}.

\begin{proof} Suppose $M$ is a countable transitive model of $ZFC+GCH$. Suppose that $G$
is $FIN(\mathfrak{c},2)$-generic over $M$ (Proposition
\ref{prop}).

Note that the following property holds:

$(\star)$ for every family $\mathcal{F}$ of Borel subsets of
$\omega^{\omega}$ with size $\aleph_1<|\mathcal{F}|<\mathfrak{c}$,
if $\bigcup \mathcal{F}=\omega^{\omega}$ then there exists
$\mathcal{F}_0\in [\mathcal{F}]^{\omega_1}$ with $\bigcup
\mathcal{F}_0=\omega^{\omega}$ (see the proof of Corollary 3.13 in
\cite{BrMi}).

Fix $\aleph_1<\kappa<\mathfrak{c}$. Let $X$ be an absolute
$\kappa$-Borel set of density $\kappa$.

Assume that there is a condensation $g$ of $X$ onto a compactum $K$. Since $X$ is an absolute $\kappa$-Borel set of
density $\kappa$ there is  a continuous bijection $f: A\rightarrow
X$ where $A$ is a closed subset of $\kappa^{\omega}$ (Theorem 5 in
\cite{Han}) and a continuous surjection $q:
\kappa^{\omega}\rightarrow A$ (Theorem 4 in \cite{st1}). Since $f$
is a continuous bijection  and $d(X)=\kappa$, $d(A)=\kappa$.

 Then we have the continuous bijection $h=g\circ
f: A\rightarrow K$ from $A$ onto $K$.

Let $\sum=[\kappa]^{\omega}\cap M$. Note that
$|\sum|<\mathfrak{c}$ since in $M$ $|\kappa^{\omega}|>\kappa$ if
and only if $\kappa$ has cofinality $\omega$, but in that case
$|\kappa^{\omega}|=|\kappa^+|<\mathfrak{c}$. Since the forcing is
c.c.c.

$M[G]\models \kappa^{\omega}=\bigcup \{Y^{\omega}: Y\in \sum \}$.

Let $\sum':=\{Y\in \sum: Y^{\omega}\cap A\neq \emptyset\}$. For
any $Y\in \sum'$ the continuous image $h(Y^{\omega}\cap A)$ (note
that $Y^{\omega}\cap A$ is Polish because $A$ is closed) is an
analytic set (a $\Sigma^1_1$ set) and, hence the union of
$\omega_1$ Borel sets in $K$ (see Ch.3, $\S$ 39, Corollary 3 in
\cite{Kur1}).

Thus $h(Y^{\omega}\cap A)=\bigcup \{B(Y,\beta): \beta<\omega_1\}$
where $B(Y,\beta)$ is a Borel subset in $K$ for each $Y\in \sum'$
 and $\beta<\omega_1$.

Let $\theta=\min\{|S|: S\subseteq\{B(Y,\beta): Y\in \sum',
\beta<\omega_1\}$ and $\bigcup S=K\}$. Note that $\theta\leq
|\sum'|\leq |\sum|<\mathfrak{c}$.

{\it Claim 1.}  $\theta\geq\kappa$.

Assume that $\theta<\kappa$. Let $S=\{B(Y_\zeta,\beta_\zeta):
\zeta\in \theta\}$. Consider a function $\phi:
\{B(Y_\zeta,\beta_\zeta): \zeta\in \theta \}\rightarrow \sum'$
such that $\phi(B(Y_\zeta,\beta_\zeta))=Y_{\xi}\in \sum'$ where
$h(Y_\xi^{\omega}\cap A)$ contains in decomposition the set
$B(Y_\zeta,\beta_\zeta)$ ($Y_{\xi}$ may be the same for different
$B(Y_1,\beta_1)$ and $B(Y_2,\beta_2)$). Since $h$ is a
condensation and $K=\bigcup\{h(Y_\xi^{\omega}\cap A):
\xi\in\theta\}$, $X=\bigcup \{f(Y^{\omega}_{\xi}):\xi\in\theta\}$
and $A\subseteq \bigcup\{Y^{\omega}_{\xi}:\xi\in\theta\}$. Let
$Q=\bigcup\{Y_{\xi}:\xi\in\theta\}$ then $Q\in [\kappa]^{\leq
\theta}$. Note that $A\subseteq
\bigcup\{Y^{\omega}_{\xi}:\xi\in\theta\}\subseteq
Q^{\omega}\subset \kappa^{\omega}$ and $w(Q^{\omega})\leq \theta$.
Since $f$ is continuous, $w(f(Q^{\omega}))\leq w(Q^{\omega})$. But
$\kappa=w(X)=w(f(Q^{\omega}))\leq w(Q^{\omega})\leq\theta$ is a
contradiction.

 Thus, $\kappa\leq \theta\leq
|\sum|<\mathfrak{c}$.

 Since $K$ is Polish, there is a continuous
surjection $p: \omega^{\omega}\rightarrow K$. Given a family
$\mathcal{F}=\{p^{-1}(B(Y_{\xi}, \beta_{\xi}): \xi<\theta \}$ of
$\theta$-many Borel sets ($\aleph_1<\theta<\mathfrak{c}$) whose
union is $\omega^{\omega}$.

 By property $(\star)$, there is a
subfamily $\mathcal{F}_0=\{F_{\alpha}:
F_{\alpha}=p^{-1}(B(Y_{\xi_{\alpha}}, \beta_{\xi_{\alpha}}):
\alpha<\omega_1\}$ of size $\omega_1$ whose union is
$\omega^{\omega}$. Then the family
$\{B(Y_{\xi_{\alpha}},\beta_{\xi_{\alpha}}) : \alpha<\omega_1\}$
of size $\omega_1$ whose union is $K$. It follows that $\theta\leq
\aleph_1$, by Claim 1, it is a contradiction.

\end{proof}

\begin{proposition} There is a forcing extension in which for any $\kappa\in (\aleph_1, \mathfrak{c})$ there exists a metric space $X$ of density
$\kappa$ such that $X$ condenses onto $[0,1]$ but any completion
$\widetilde{X}$ of $X$ is not condensed onto a compactum.
\end{proposition}

\begin{proof}
Consider the forcing extension in Theorem \ref{th3}. Let
$[0,1]=\coprod \{A_{\alpha}: \alpha\in \kappa\}$ be a partition
such that $A_{\alpha}$ is dense in $[0,1]$ for each $\alpha\in
\kappa$. Let $X=\bigoplus \{A_{\alpha}: \alpha\in \kappa\}$. It is clear
that $X$ condenses onto $[0,1]$, but any completion
$\widetilde{X}$ of $X$ is an absolute Borel set of density
$\kappa$. Hence, by Theorem \ref{th3}, $\widetilde{X}$ is not
condensed onto a compactum.
\end{proof}

\section{Appendix}

Recently,  a good work has been done to partition Polish spaces into Borel sets [2-6].
In \cite{Br2}, W. Brian defined the {\it Borel partition spectrum}, denote $\mathfrak{sp}(Borel)$, as follows:

$\mathfrak{sp}(Borel)=\{|\mathcal{P}|: \mathcal{P}$ is a partition of $\mathbb{R}$ into uncountable many Borel sets $\}$.

Let us by reviewing what is known about  $\mathfrak{sp}(Borel)$.

\medskip

(1) $\aleph_1\in \mathfrak{sp}(Borel)$.

\medskip

(2) $\mathfrak{c}=\max(\mathfrak{sp}(Borel))$.

\medskip

(3)  $\mathfrak{sp}(Borel)$ is closed under singular limits.

\medskip

(4) It is consistent with arbitrary values of $\mathfrak{c}$ that  $\mathfrak{sp}(Borel)=\{\aleph_1,\mathfrak{c}\}$.

\medskip

(5) It is consistent with arbitrary values of $\mathfrak{c}$  that  $\mathfrak{sp}(Borel)=[\aleph_1,\mathfrak{c}]$.

\medskip

(6)  Suppose $C$ is a set of uncountable cardinals such that:

\medskip

 \, \, (a) $C$ is countable,

 \, \, (b) $\aleph_1\in C$,

 \, \, (c) $C$ has a maximum, and $\max(C)^{\aleph_0}=\max(C)$,

\, \, (d) $C$ is closed under singular limits, and

\, \, (e) if $\lambda\in C$ and $cf(\lambda)=\omega$, then $\lambda^+\in C$.

\medskip

Assuming $GCH$ holds up to $\max C$, there is a $ccc$ forcing extension
in which $C=\mathfrak{sp}(Borel)$.

\medskip

The following theorem unexpectedly points to the connection  between those $\kappa$ for which absolute $\kappa$-Borel set, containing a closed subspace of the Baire space
of weight $\kappa$, condenses onto a compactum, and those $\kappa$ for which there is a partition of the real line into $\kappa$ Borel sets, i.e. $\kappa\in \mathfrak{sp}(Borel)$.

Given a topological space X, define

$\mathfrak{par}(X)=\min\{|P|: P$ is a partition of $X$ into Polish spaces $\}$.

\begin{theorem}\label{th5} Let $\kappa$ be an uncountable cardinal with $\kappa=\mathfrak{par}(\kappa^{\omega})$. Then
the following are equivalent:

(1) $\kappa\in \mathfrak{sp}(Borel)$.

(2) An absolute $\kappa$-Borel set, containing a closed subspace $\kappa^{\omega}$, condenses onto a compactum.

\end{theorem}

\begin{proof}   Assume that $\kappa\in \mathfrak{sp}(Borel)$. Then, by the Banakh-Plichko Theorem \cite{BaPl}, any Banach space of density $\kappa$ condenses onto the Hilbert cube. By the Medvedev Theorem, any absolute $\kappa$-Borel set $X$ of density $\kappa$,
containing a closed subspace of the Baire space $\kappa^{\omega}$
of weight $\kappa$, condenses onto a Banach space of weight
$\kappa$ (Theorem 3 in \cite{Med}). It follows that $X$ condenses onto the Hilbert cube.

Assume that there is a condensation $g$ of an absolute $\kappa$-Borel set $X$ onto a compactum $K$. Since $X$ is an absolute $\kappa$-Borel set of
density $\kappa$ there is  a continuous bijection $f: A\rightarrow
X$ where $A$ is a closed subset of $\kappa^{\omega}$ (Theorem 5 in
\cite{Han}) and a continuous surjection $q:
\kappa^{\omega}\rightarrow A$ (Theorem 4 in \cite{st1}). Since $f$
is a continuous bijection  and $d(X)=\kappa$, $d(A)=\kappa$.  Then we have the continuous bijection $h=g\circ
f: A\rightarrow K$ from $A$ onto $K$.

Since $\kappa=\mathfrak{par}(\kappa^{\omega})$, the set $A$ can be partitioned into $\kappa$ Polish spaces. Let $A=\bigsqcup_{\alpha<\kappa} P_{\alpha}$ where $P_{\alpha}$ is Polish for every $\alpha\in \kappa$.
By a theorem of Lusin and Suslin, $h(P_{\alpha})$ is a Borel subset of $K$ for every $\alpha\in \kappa$. By Lemma 3.1 in \cite{Br5}, $\kappa\in \mathfrak{sp}(Borel)$.

\end{proof}

By Theorem \ref{th5} and  Theorem 3.2 in \cite{Br2},  we have

\begin{corollary} {\it Given any $A\subseteq \mathbb{N}$ with $1\in A$, there is a forcing extension in which every  absolute $\aleph_n$-Borel set, containing a closed subspace of the Baire space
of weight $\aleph_n$, condenses onto a compactum if, and only if, $n\in A$ }.
\end{corollary}

In particular, by Theorem \ref{th5} and Corollaries 3.3 and 3.4 in \cite{Br2}, we have the following results.

\begin{corollary}{\it Given any $A\subseteq \mathbb{N}$, $1\in A$, there is a forcing extension in which an absolute $\kappa$-Banach space $X$ of density $\kappa$ condenses onto the Hilbert cube if, and only if, $\kappa\in \{\aleph_n:n\in A\}\cup \{\aleph_{\omega}, \aleph_{\omega+1}= \mathfrak{c}\}$}.
\end{corollary}

\begin{corollary}{\it Given any finite $A\subseteq \mathbb{N}$, $1\in A$, there is a forcing extension in which an absolute $\kappa$-Banach space $X$ of density $\kappa$ condenses onto the Hilbert cube if, and only if, $\kappa\in \{\aleph_n:n\in A\}$}.
\end{corollary}

\bigskip

{\bf Acknowledgment.} I would like to thank William Brian and the
referee for careful reading and valuable comments.

\bibliographystyle{model1a-num-names}
\bibliography{<your-bib-database>}

\end{document}